\documentclass[11pt,a4paper]{amsart}
\usepackage[all]{xy}
\SelectTips{cm}{}
\calclayout

\usepackage{amscd,mathtools,amssymb,amsopn,amsmath,amsthm,graphics,mathrsfs,accents}
\usepackage[colorlinks=true,linkcolor=red,citecolor=blue]{hyperref}
\usepackage[mathscr]{euscript}
\usepackage[usenames]{color}

\newcommand{\rt}{\rightarrow}

\newcommand{\lrt}{\longrightarrow}
\newcommand{\lf}{\leftarrow}
\newcommand{\llf}{\longleftarrow}

\newcommand{\st}{\stackrel}

\newcommand{\Md}{\mathsf{Mod}}

\newcommand{\im}{\mathsf{im}}
\def\ker{\operatorname{\mathsf{ker}}}

\newcommand{\Q}{\mathcal{Q} }
\newcommand{\Z}{\mathbb{Z} }

\newcommand{\C}{\mathscr{C} }

\newcommand{\G}{\mathcal{G} }

\newcommand{\m}{\mathfrak{m}}

\newcommand{\p}{\mathfrak{p} }

\def\proj{\operatorname{\mathsf{proj}}}
\def\inj{\operatorname{\mathsf{inj}}}

\def\Ext{\operatorname{{\mathsf{Ext}}}}

\def\hom{\operatorname{{\mathsf{Hom}}}}
\def\uhom{\operatorname{\underline{\mathsf{Hom}}}}

\def\Tor{\operatorname{\mathsf{Tor}}}
\def\hom{\operatorname{\mathsf{Hom}}}

\def\X{\mathcal{X}}
\def\O{\mathcal{O}}

\def\G{\mathcal{G}}

\def\md{\operatorname{\mathsf{mod}}}

\def\syz{\mathsf{\Omega}}
\def\coh{\operatorname{\mathsf{coh}}}

\def\x{\mathsf {X}}

\def\tr{\mathsf {Tr}}

\def\L{\mathcal L}
\def\p{\mathcal P}
\def\A{\mathcal A}

\def\U{\mathcal U}

\def\ruf{\operatorname{\mathsf{RUF}}}
\def\luf{\operatorname{\mathsf{LUF}}}

\def\al{\operatorname{\boldsymbol{\alpha}}}
\def\be{\operatorname{\boldsymbol{\beta}}}
\def\ga{\operatorname{\boldsymbol{\gamma}}}

\def\ep{\operatorname{\boldsymbol{\epsilon}}}
\def\de{\operatorname{\boldsymbol{\delta}}}
\def\et{\operatorname{\boldsymbol{\eta}}}

\newcommand{\nat}{{\mathsf{Nat}}}

\newtheorem{theorem}{Theorem}[section]

\newtheorem{cor}[theorem]{Corollary}

\newtheorem{lem}[theorem]{Lemma}
\newtheorem{prop}[theorem]{Proposition}

\theoremstyle{definition}
\newtheorem{dfn}[theorem]{Definition}
\newtheorem{example}[theorem]{Example}

\newtheorem{rem}[theorem]{Remark}
\newtheorem{s}[theorem]{}

\theoremstyle{plain}

\theoremstyle{definition}

\numberwithin{equation}{section}

\begin{document}

\title[On the natural transformations of extension functors]
{On the natural transformations of extension functors}
\author[Bahlekeh and Salarian]{Abdolnaser Bahlekeh and Shokrollah Salarian}

\address{Department of Mathematics, Gonbad-Kavous University, Postal Code:4971799151, Gonbad-Kavous, Iran}
\email{bahlekeh@gonbad.ac.ir}

\address{Department of Pure Mathematics, Faculty of Mathematics and Statistics,
University of Isfahan, P.O.Box: 81746-73441, Isfahan,
 Iran and \\ School of Mathematics, Institute for Research in Fundamental Science (IPM), P.O.Box: 19395-5746, Tehran, Iran}
 \email{Salarian@ipm.ir}


\thanks{}

\begin{abstract}
Assume that $n$ is a non-negative integer and $\C$  an $n$-Frobenius category.
This paper is concerned with the natural transformations between extension functors on $\C$.  The  main result of the paper indicates that  for any pair of objects $M, N\in \C$, the group of all natural transformations from $\Ext^{n+1}_{\C}(N, -)$ to $\Ext^{n+1}_{\C}(M, -)$ is isomorphic to the quotient group $\Ext^n_{\C}(M, \syz^nN)/{\p}$, where $\p$ is the subgroup consisting of those Yoneda extensions of length $n$ arising as a push-out along a morphism $P\rt\syz^nN$, with $P$ $n$-projective. This result is fruitful from the point of view that, $n$-Frobenius categories may not have projective objects.
\end{abstract}
\maketitle

\tableofcontents

\section{Introduction}
Assume that $\A$ is an abelian category and $\C$ an extension-closed full additive subcategory of $\A$. So, it will be an exact category in the sense of Quillen, see \cite[Lemma 10.20]{buh}. One of the most important results concerning natural transformations of extension functors on $\C$, is due to  Hilton-Rees \cite{hr}. To be precise, assume that $\C$ has enough projective objects, and $f: M\rt N$ is a morphism in $\C$. Then $\Psi_f:\Ext^1_{\C}(N, -)\lrt\Ext^1_{\C}(M, -)$, in which for any object $T\in\C$ and $\al\in\Ext^1_{\C}(N, T)$, $\Psi_f(T)(\al)=\al f$, is a natural transformation of extension functors, where $\al f$ is the pull-back of $\al$ along $f$. Now, the Hilton-Rees theorem says that the assignment
$$\uhom_{\C}(M, N)\lrt\nat(\Ext^1_{\C}(N, -), \Ext^1_{\C}(M, -)):~~~\bar{f}\mapsto\Psi_f$$is an isomorphism of abelian groups,  see  \cite[Theorem 7]{mart}. Here $\uhom_{\C}(M, N)$ is the group of morphisms from $M$ to $N$ modulo the morphisms factoring through projectives, and  $\nat(\Ext^1_{\C}(N, -), \Ext^1_{\C}(M, -))$ stands for the group of all natural transformations from $\Ext^1_{\C}(N, -)$ to $\Ext^1_{\C}(M, -)$. This, in particular, yields that the latter forms a set. 

The main goal of this paper is to study the natural transformations between extension functors over $n$-Frobenius categories, where $n$ is a non-negative integer.  The notion of an $n$-Frobenius category and its stabilization, have been introduced and studied in \cite{bfss, bfss1}.  Recall that an exact category $\C$ is said to be an $n$-Frobenius category,  if  $\C$ has enough $n$-projective and $n$-injective objects and these two classes of objects coincide, where the notions of $n$-projectivity and $n$-injectivity are defined via the vanishing of $\Ext^{n+1}_{\C}(-, -)$, $(n+1)$-th Yoneda $\Ext$. So, 0-Frobenius categories are exactly Frobenius categories in the usual sense. It is worth noting that $n$-Frobenius categories need not have projective objects, see Example \ref{ex}(3, 6). 
{ Indeed, $n$-projective objects, and so, $n$-Frobenius categories that do not have projective objects, appear naturally in the category of sheaves over noetherian schemes, see \cite{bfss, ee, har, or}. Assume that  $(\x, \O_\x)$ is a  semi-separated noetherian scheme  of finite Krull dimension. Then by  a result of Orlov \cite[Lemma 1.12]{or}, there exists a non-negative integer $n$ such that  any locally free sheaf of finite rank is an $n$-projective object of $\coh(\x)$, the category of all coherent sheaves. Particularly, the category $\C(\L)$ consisting of all syzygies of complete resolutions of locally free sheaves of finite rank, is an $n$-Frobenius subcategory of $\coh(\x)$, for details see Example  \ref{xx1}.

Now assume that $\C$ is an $n$-Frobenius category.  This paper deals with the group of natural transformations $\nat(\Ext^{n+1}_{\C}(N, -), \Ext^{n+1}_{\C}(M, -))$, where $M, N$ are arbitrary objects of $\C$. In this direction, first we show that for any  object $T\in\C$, and $\al\in\Ext^{n+1}_{\C}(N, T)$, there exists a syzygy sequence $\de'_N\in\Ext^n_{\C}(N, {\syz'}^nN)$ with respect to $n$-projective objects, and  an extension $\alpha_1\in\Ext^1_{\C}({\syz'}^nN, T)$ such that there is a morphism of extensions $\alpha_1\de'_N\lrt\al$ with fixed ends, see Proposition \ref{ll}. Next, for a given $\ga\in\Ext^n_{\C}(M, \syz^nN)$, with $\de_N\in\Ext^n_{\C}(N, \syz^nN)$ a syzygy sequence using $n$-projectives, we prove that there exists an element $\ga'\in\Ext^n_{\C}(M, {\syz'}^nN)$ such that $\Psi_{\ga}:\Ext^{n+1}_{\C}(N, -)\lrt\Ext^{n+1}_{\C}(M, -)$, sending each $\al\in\Ext^{n+1}_{\C}(N, T)$ to $\alpha_1\ga'$, is a natural transformation of extension functors, see Propositions \ref{3p1} and \ref{p33}. Now our main theorem reads as follows.

\begin{theorem}\label{main2} (Theorem \ref{nat}) Keeping the notation above,  the correspondence
$$\Psi: \Ext^n_{\C}(M, \syz^nN)\lrt{\nat}(\Ext^{n+1}_{\C}(N, -), \Ext^{n+1}_{\C}(M, -))~:~~\ga\mapsto \Psi_{\ga}$$
induces an isomorphism of abelian groups
$$\Ext^n_{\C}(M, \syz^nN)/{\p}\st{\cong}\lrt{\nat}(\Ext^{n+1}_{\C}(N, -), \Ext^{n+1}_{\C}(M, -)).$$
Here $\p$ is the subgroup consisting of those extensions arising as a push-out along a morphism $f:P\rt\syz^nN$, with $P$ $n$-projective.
\end{theorem}
One should note that since $\C$ has also enough $n$-injectives and this class coincides with $n$-projectives, a dual manner yields  that  the group $\Ext^n_{\C}(\syz^{-n}M, N)/{\p}$ is isomorphic to $\nat(\Ext^{n+1}_{\C}(-, M), \Ext^{n+1}_{\C}(-, N)),$ see Theorem \ref{th3}.  In particular, as $n$-Frobenius categories are $k$-Frobenius, for any integer $k>n$, the aforementioned results  imply the following corollary.
}

\begin{cor}\label{123}(Corollary \ref{ccc}) Let $\C$ be an $n$-Frobenius category and let $M, N$ be objects of $\C$. Then for any  integer $k\geq n$ and a syzygy sequence $\de_N\in\Ext^k_{\C}(N, \syz^kN)$ with respect to $n$-pojectives,  one has   isomorphisms of abelian groups
\[\begin{array}{lllll}
\Ext^k_{\C}(M, \syz^kN)/{\p} & \cong \nat(\Ext^{k+1}_{\C}(N, -), \Ext^{k+1}_{\C}(M, -))\\
& \cong \nat(\Ext^{k+1}_{\C}(-, M), \Ext^{k+1}_{\C}(-, N)).
\end{array}\]
Here $\p$ is the subgroup consisting of those extensions arising as a push-out along a morphism $f:P\rt\syz^kN$, with $P$ $n$-projective.
\end{cor}
It should be mentioned that since 0-Frobenius categories are Frobenius categories in the usual sense, the subgroup $\p$ of $\Ext^0_{\C}(M, \syz^0N)=\hom_{\C}(M, N),$ will be those morphisms $f:M\rt N$ that factor through a projective object. So, the case $n=0$ of the above result recovers the Hilton-Rees theorem for Frobenius categories,  see Corollary \ref{c44}. From this point of view, Corollary \ref{123} can be regarded as a vast generalization of the Hilton-Rees theorem  for $n$-Frobenius categories, which need not have projective objects.
 We also  illustrate some explicit examples of $n$-Frobenius categories for which the above isomorphisms are satisfied. Particularly, if $(\x, \O_{\x})$ is a semi-separated noetherian scheme of finite Krull dimension, then  for any two objects $X, Y\in\C(\L)$ and any integer $k\geq n$, the group of natural transformations $\nat(\Ext^{k+1}_{\C(\L)}(Y, -), \Ext^{k+1}_{\C(\L)}(X, -))$ will be isomorphic to $\Ext^k_{\C(\L)}(X, \syz^kY)/{\p}$, where $\syz^kY$ is obtained by using locally free sheaves of finite rank, see Example \ref{sep}.

Throughout the paper, $\C$ is an extension-closed full additive subcategory of an abelian category $\A$.  So, it will be an exact category. We always assume that $n$ is a non-negative integer.
 

\section{Preliminaries}

In this section, we recall some notions that will be used in the paper. Particularly, we will remember some facts on extensions briefly and also recall the notion of $n$-Frobenius categories. Moreover, several examples of such categories are provided.

\begin{s}\label{gener}Since $\C$ is an extension-closed subcategory of the abelian category $\A$, the exact structure of $\A$ is inherited by $\C$. Namely, conflations are those short exact sequences in $\A$ with terms in $\C$,  see \cite[Lemma 10.20]{buh}.  We will make use  the following considerations.
\begin{itemize}
\item[(a)] Throughout the paper, instead of conflations, we use the term `extensions'. Also, for simplicity, each extension of the form $0\rt M\rt X\rt N\rt 0$ in $\C$ will be denoted by $M\rt X\rt N$.
\item[(b)] Assume that $n\geq 1$ is an integer. By an extension of length $n$ in $\C$, we mean an exact sequence of a form $\et: 0\rt A\rt X_{n-1}\rt\cdots\rt X_0\rt B\rt 0$ in $\A$, which is obtained by splicing $n$ composable extensions in $\C$. Also,  $A$ and $B$ are called the left and right ends of $\et$, respectively. Such a sequence $\et$ will be denoted by $ A\rt X_{n-1}\rt\cdots\rt X_0\rt B$. If there is no ambiguity, extensions of length $n$ will just be called extensions.
\item[(c)] The set of all equivalence classes of extensions of length $n$ in $\C$, with  $A$ and $B$ being the left and right ends, will be denoted by $\Ext^n_{\C}(B, A)$. We should remind that the equivalence relation is generated by the relation that $\et\sim\et'$, if there is a morphism of extensions $\et\lrt\et'$, with fixed ends. Namely, there is a commutative diagram in $\C$
{\footnotesize\[\xymatrix{\et: A\ar@{=}[d]\ar[r] & X_{n-1}\ar[d]\ar[r] & \cdots\ar[r] & X_0\ar[r]\ar[d] & B\ar@{=}[d]\\ \et': A\ar[r] & X'_{n-1}\ar[r] & \cdots\ar[r] & X'_0\ar[r] & B.}\]}It should be noted that, by the convention, $\Ext^0_{\C}(B, A):=\hom_{\C}(B, A)$. It is known that for any integer $n\geq 1$, $\Ext^n_{\C}(B, A)$, with respect to the Baer sum operation, admits a structure of an abelian group. Indeed, for any pair of elements $\al, \be\in\Ext^n_{\C}(B, A)$, we have $\al+\be=\nabla_A(\al\oplus\be)\Delta_B$. For more details, see \cite[Chapter VII]{mit} and also \cite[Section 9]{hs}. \item[(d)] For given  elements $\et\in\Ext^t_{\C}(B, A)$ and $\et'\in\Ext^{t'}_{\C}(C, B)$, one may splice these together and get  the composition $\et\et'$, which is  an element of $\Ext^{t+t'}_{\C}(C, A)$.
\end{itemize}
\end{s}

\begin{rem}Assume that $f:A\rt A'$ and $g:B\rt B'$ are morphisms in $\C$. Then for any integer $n\geq 0$, the correspondences  $$\Ext^n(B, f):\Ext^n_C(B, A)\lrt\Ext^n_{\C}(B, A') :~~\al\mapsto f\al ~~~and$$
$$\Ext^n(g, A):\Ext^n_{\C}(B', A)\lrt\Ext^n_{\C}(B, A) :~~\be\mapsto \be g$$
are morphisms of abelian groups. Here $f\al$ is the push-out of $\al$ along $f$, and $\be g$ stands for the pull-back of $\be$ along the morphism $g$. Note that, as $\Ext^0_{\C}(-, -)=\hom_{\C}(-, -)$, these operations can be interpreted as the composition of morphisms in $\C$.
\end{rem}

\begin{theorem}\label{long}Let $\ep: A\st{f}\rt B\st{g}\rt C$ be an extension in $\C$. Then for any object $X\in\C$ and any integer $n>0$, there exists an exact sequence of groups
$$\Ext^{n-1}_{\C}(X, C)\st{\Delta^{n-1}_X}\lrt\Ext^n_{\C}(X, A)\st{\Ext^n(X, f)}\lrt\Ext^n_{\C}(X, B)\st{\Ext^n(X, g)}\lrt\Ext^n_{\C}(X, C)\st{\Delta^n_X}\lrt\Ext^{n+1}_{\C}(X, A),$$where $\Delta^n_{X}(\ga)=\ep\ga$, for any $\ga\in\Ext^n_{\C}(X, C)$.
\end{theorem}
\begin{proof}See \cite[Chapter VII, Theorem 5.1]{mit}.
\end{proof}

\begin{dfn}\label{x}Assume that $\X$ is a class of objects in $\C$ which is closed under finite direct sums. Then for any two objects $A, B\in\C$ and any integer $t\geq 0$, we say that an element $\al\in\Ext^t_{\C}(B, A)$, lies in $\X(B, A)$, if there exists a morphism $f:X\rt A$ in $\C$, with $X\in\X$, and an element $\be\in\Ext^t_{\C}(B, X)$ such that $\al=f\be$. \\
It should be noted that, since $\X$ is closed under finite direct sums, one may easily observe that  $\X(B, A)$ is an additive subgroup of $\Ext^t_{\C}(B, A)$. Moreover, if $t=0$, then $\X(B, A)$ will be the subgroup of $\hom_{\C}(B, A)$ consisting of those morphisms that factor through some objects in $\X$.
\end{dfn}

\begin{dfn}(\cite[Definition 2.3]{bfss}) The exact category $\C$ is called $n$-Frobenius, if
it has enough $n$-projectives and $n$-injectives and these two classes of objects coincide. Recall that
a given object $P\in\C$ (resp. $I\in\C$) is said to be {\it $n$-projective} (resp. {\it $n$-injective}), if $\Ext^i_{\C}(P, X)=0$ (resp. $\Ext^i_{\C}(X, I)=0$) for all integers $i>n$ and all objects $X\in\C$.
\end{dfn}
The above definition yields that 0-Frobenius categories are exactly Frobenius categories in the usual sense, and also, an $n$-Frobenius category is a $k$-Frobenius cayegory, for any integer $k>n$.  Moreover, \cite[Theorem 2.15]{bfss} indicates that any abelian category with non-zero $n$-projective objects admits a non-trivial $n$-Frobenius subcategory.

As noted in the introduction, the notion of $n$-Frobenius categories was introduced and studied in \cite{bfss}. The motivating reason for studying this notion was  the existence of categories that rarely have enough projectives, or even, have no projective objects, and so, they would be never Frobenius. However, this kind of categories which naturally  appear in the realm of sheaves over noetherian schemes, admit non-trivial $n$-projective objects, for some integer $n\geq 1$, for more details see \cite{bfss, ee, har, or}.

\begin{rem}\label{ppp} Keeping the notation of Definition \ref{x}, in this paper, we consider the following two special cases:
\begin{itemize}
\item If $\C$ has enough projective objects, then we take $\X:=\proj\C$, the class of all projective objects of $\C$.
\item If $\C$ is an $n$-Frobenius category, then we set $\X:=n$-$\proj\C$, the class of all $n$-projective objects of $\C$.
\end{itemize}
In both cases, for any pair of objects $ A, B\in\C$, $\X(B, A)$, which will be denoted by $\p(B, A)$, is a subgroup of $\Ext^t_{\C}(B, A)$. If there is no confusion, we denote $\p(-, -)$  by $\p$, and extensions which lie in $\p$ will be called $\p$-extensions.
\end{rem}
There exists a wide family of interesting examples of Frobenius categories  in the literature, see for instance \cite{stable, ha, bfns, berg}.
Below we give several examples of $n$-Frobenius categories.
\begin{example}\label{ex}
Let $(R, \m)$ be a $d$-dimensional commutative noetherian local ring.\\ (1) The category $\G^{<\infty}$ consisting of all $R$-modules of finite Gorenstein projective dimension is a $d$-Frobenius category, and its $d$-projective objects are indeed the modules of finite projective dimension, see \cite[Example 2.16]{bfss}.\\
{(2) If $R$ is Gorenstein, then the category $\md R$ of all finitely generated $R$-modules, is a $d$-Frobenius category. Particularly, if $d>0$, then injective modules can not be finitely generated, and so, $\md R$ does not have injective objects.\\ (3) Assume that $R$ is  complete  and $\C$ is the class of all artinian $R$-modules. It is known that $D(-):=\hom_{R}(-, E(R/{\m})):\md R\lrt\C$ is a duality. So, according to the previous statement, if $R$ is Gorenstein, then $\C$ will be a $d$-Frobenius category, as well. In particular, if $d>0$, then $\C$ does not have projective objects, because $\md R$ does not have injective objects.}\\ (4) The class $\mathsf{GF}(R)$ of all Gorenstein flat $R$-modules, is a $d$-Frobenius category, and its $d$-projective objects are flat modules, see \cite[Example 2.4(2)]{bfss1}.\\ (5) Assume that $R$ is Gorenstein and $\Q$ a finite acyclic quiver. Then the category  $\G$ of all lattices over $R\Q$ is a 1-Frobenius subcategory of $R\Q$-$\Md$, in which its 1-projective objects are those $R\Q$-modules $M$ which are projective over $R$. We should point out that $\G$ would not be a Frobenius category, if $\Q$ has at least two vertices. Recall that by a lattice over $R\Q$, we mean an $R\Q$-module  which is finitely generated Gorenstein projective over $R$, see \cite[Example 2.4(3)]{bfss1}.\\
(6) Assume that $\coh(\mathbf{P^{1}})$ is the category of  coherent sheaves over the projective line $\mathbf{P^{1}}(R)$, and $\C$
is its full subcategory whose objects are of the form $ \G\equiv G_1 \rt G \lf G_2$, where $G_1$ (resp. $G_2$) is Gorenstein projective over $R[x]$ (resp. $R[x^{-1}]$). Then one may see that $\C$ is a 1-Frobenius category which has no projective objects, see  \cite[Corollary 2.3]{ee}.
\end{example}

\begin{example}\label{xx1}
Let $(\x, \O_\x)$ be a  semi-separated noetherian scheme  of finite Krull dimension. Then  a result of
Orlov \cite[Lemma 1.12]{or} guarantees the existence of a non-negative integer $n$ such that, for any locally free sheaf $L$ of finite rank and any coherent sheaf $G$, $\Ext^{i>n}_{\coh(\x)}(L, G) = 0$. Assume that $\L$ is the category of all locally free sheaves of finite rank, and $\C(\L)$ is the subcategory of $\coh(\x)$, whose objects are  all syzygies of complete resolutions of locally free sheaves of finite rank. It is known that  $\C(\L)$ is an $n$-Frobenius category, with $n$-$\proj\C(\L)=\L$, see \cite[Proposition 2.19]{bfss}.

Recall that an acyclic complex of locally free sheaves of finite rank,
 $$\mathbf{P^{\bullet}}:\cdots\lrt P^{i-1}\st{d^{i-1}}\lrt P^i\st{d^i}\lrt P^{i+1}\lrt\cdots$$
 is said to be a complete resolution of locally free sheaves, if $\Ext^{n+1}_{\coh(\x)}(\im d^i, Q)=0$, for any $i\in\Z$ and $Q\in \L$.  
 
Assume further that  $(\x, \O_{\x})$ is a Gorenstein scheme, that is,  all of its local rings are Gorenstein local rings.  Since it is a semi-separated noetherian scheme  of finite Krull dimension, one may  deduce that $\inj.\mathsf{dim}_{\O_{\x}}\O_{\x}$ is finite. Thus, it is not hard to see that  any acyclic complex of locally free sheaves of finite rank is a complete resolution.
\end{example}

\begin{s}\label{ss}Assume that $\C$ is an $n$-Frobenius category.\\
(1) Since $\C$ has enough $n$-projective objects, for any object $M\in\C$ and an integer $t\geq 1$, one may get a syzygy sequence  $\syz^tM\rt P_{t-1}\rt\cdots\rt P_0\rt M$, where each $P_i$ is  $n$-projective.  Also, $\syz^tM$ is called a $t$-th syzygy of $M$. Syzygy sequences are also called unit extensions.\\ (2)  For a given object $X\in\C$, the class of all  unit  extensions (of length $n$)  ending at  $X$ (resp.  beginning with  $X$) is denoted by $\U_n(X)$ (resp.  {$\U^n(X)$}).\\ (2) Assume that $t\geq 1$ is an integer and $\delta, \delta' \in \U^t(N)$ (resp. $\delta_1, \delta_2 \in\U_t(N)$). Then there exist deflations (resp. inflations) $a, a'$ with $n$-projective kernels (resp. cokernels) such that $\delta a=\delta'a'$ (resp. $a\delta_1=a'\delta_2$), see \cite[Proposition 5.5]{bfss}. In this situation, $[\delta a, \delta'a']$  (resp. $[a\delta_1, a'\delta_2]$) is called a co-angled pair (resp. an angled pair). In some cases, we denote it by $\delta\st{a}\llf\delta''\st{a'}\lrt\delta'$ (resp. $\delta_1\st{a}\lrt\delta''\st{a'}\llf\delta_2$), where $\delta''=\delta a$  (resp. $\delta''=a\delta_1$).\\ (3) Suppose that an element $\ga\in\Ext^t_{\C}(M, \syz^tN)$ is given, where $t\geq 1$ is an integer. So there exist unit extensions $\delta\in\U^t(\syz^t N)$, $\delta'\in\U_t(M)$ and morphisms $f,g$ in $\C$ such that $\ga=\delta f$ and $\ga=g\delta'$, see \cite[Proposition 5.2]{bfss}. These are called a right and a left unit factorization of $\ga$, respectively, and abbreviated by $\ruf$ and $\luf$.
\end{s}

\begin{s}{\sc $n$-$\Ext$-phantom (invertible) morphisms.} Assume that $\C$ is an $n$-Frobenius category and $f:M\rt N$ is a morphism in $\C$.\\ (1) For a given object $X\in\C$, there exist induced maps $\Ext^n_{\C}(X, M)/{\p}\rt\Ext^n_{\C}(X, N)/{\p}$ and $\Ext^n_{\C}(N, X)/{\p}\rt\Ext^n_{\C}(M, X)/{\p}$ that assign  each object $\ga+\p$ to ${f\ga}+\p$ and ${\ga f}+\p$, respectively.\\ (2)
$f$ is called an $n$-$\Ext$-invertible morphism, provided that   for any object $X\in\C$, the induced morphism $\Ext^n_{\C}(X, M)/{\p}\rt\Ext^n_{\C}(X, N)/{\p}$ is an isomorphism. It can be seen that,  this is equivalent to saying that $\Ext^n_{\C}(N, X)/{\p}\rt\Ext^n_{\C}(M, X)/{\p}$) is also an isomorphism,  see \cite[Corollaries 3.4 and 4.8]{bfss}. These, in particular, imply that $f$ is an $n$-$\Ext$-invertible morphism if and only if the natural transformation $\Ext^{n+1}(-, f):\Ext^{n+1}_{\C}(-,M)\lrt\Ext^{n+1}_{\C}(-,N)$ (or equivalently, $\Ext^{n+1}(f, -)$) is an equivalence of functors. 
It should be pointed out that $n$-$\Ext$-invertible morphisms are called quasi-invertible morphisms in \cite{bfss}. Such morphisms are also called quasi-isomorphisms in \cite{hr}.\\
 (3) $f$ is said to be an $n$-$\Ext$-phantom morphism, if   for any object $X\in\C$, the induced morphism $\Ext^n_{\C}(X, M)/{\p}\rt\Ext^n_{\C}(X, N)/{\p}$  (or equivalently, $\Ext^n_{\C}(N, X)/{\p}\lrt\Ext^n_{\C}(M, X)/{\p}$) is zero, see \cite[Theorem 6.7 and Definition 6.8]{bfss}. By \cite[Corollary 6.9]{bfss},  $f$ is an $n$-$\Ext$-phantom morphism, if and only if  the natural transformation $\Ext^{n+1}(-, f):\Ext^{n+1}_{\C}(-,M)\lrt\Ext^{n+1}_{\C}(-,N)$ (or equivalently, $\Ext^{n+1}(f, -)$) vanishes.  
\end{s}

\begin{example}(1) Over an $n$-Frobenius category $\C$, any morphism
with $n$-projective kernel and cokernel is an  $n$-$\Ext$-invertible morphism, see \cite[Corollary 3.6]{bfss}. Also, a morphism that factors through an $n$-projective object is also an $n$-$\Ext$-phantom morphism.\\ (2)  Over Frobenius categories (or equivalently, 0-Frobenius categories), 0-$\Ext$-phantom morphisms are those morphisms that factor through projective objects, see \cite[Proposition 2.11(1)]{bfss1}.\\ (3) If $\C$ is a Frobenius category, then a given morphism $f: M\rt N$ in $\C$ is a 0-$\Ext$-invertible if  and only if there exist $P,Q\in\proj\C$ such that $M\oplus Q\st{{{\tiny {\left[\begin{array}{ll} f & g_1 \\ l & g_2 \end{array} \right]}}}}\lrt N\oplus P$ is an isomorphism, see \cite[Proposition 2.11(2)]{bfss1}.
\end{example}

\begin{rem}\label{invert}Assume that $\C$ is an $n$-Frobenius category and  an element $\ga\in\Ext^n_{\C}(M, N)$ is given. It is known that for any $n$-$\Ext$-invertible morphism $a:M'\rt M$, $\ga$ is a $\p$-extension if and only if $\ga a$ is a $\p$-extension. Similarly, $b\ga$ is a $\p$-extension if and only if $\ga$ is so, whenever $b:N\rt N'$ is an $n$-$\Ext$-invertible morphism, see \cite[Proposition 4.9]{bfss}.
\end{rem}

\section{A generalization of the Hilton-Rees theorem}
Assume that the exact category $\C$ has enough projective objects. The purpose of this section is to give a generalization of the Hilton-Rees theorem. Precisely, it will be observed that for any non-negative integer $n$ and any pair of objects $M, N\in\C$, the group $\nat(\Ext_{\C}^{n+1}(N, -), \Ext^{n+1}_{\C}(M, -))$ is isomorphic to $\Ext^n_{\C}(M, \syz^nN)/{\p}$, where $\p$ denotes the subgroup consisting of those extensions arising as a push-out along a morphism $P\rt\syz^nN$, with $P$ projective. We should emphasize that, as pointed out us by an expert, this result can be proved easily by using Yoneda embedding and the dimension shifting argument. However, the proof presented here works for the main theorem of the paper, Theorem \ref{main2}, where $\C$ is an $n$-Frobenius category which does not have necessarily enough projective objects.

Throughout this section, we always assume that $\C$ has enough projective objects and $n$ is a non-negative integer. Let us begin with an auxiliary result.

\begin{prop}\label{1}Let $\al$ be an element of $\Ext^{n+1}_{\C}(N, T)$. Then for a given syzygy sequence $\de_N\in\Ext^n_{\C}(N, \syz^nN)$, there exists an element $\alpha_1\in\Ext^1_{\C}(\syz^nN, T)$ such that there is a morphism $\alpha_1\de_N\lrt\al$ with fixed ends.
\end{prop}
\begin{proof}Decompose $\al: T\rt X_n\rt\cdots\rt X_0\rt N$ into two extensions $\alpha':T\rt X_n\rt L$ and $\alpha'': L\rt X_{n-1}\rt\cdots\rt X_0\rt N$. Now assuming $\de_N:\syz^nN\rt P_{n-1}\rt\cdots\rt P_0\rt N$, where each $P_i$ is projective, one may obtain the following commutative diagram with exact rows: {\footnotesize\[\xymatrix{\de_N:\syz^{n}N\ar[d]_{f}\ar[r] & P_{n-1}\ar[d]\ar[r] & \cdots\ar[r] & P_0\ar[r]\ar[d] & N\ar@{=}[d]\\
\alpha'':L\ar[r] & X_{n-1}\ar[r] & \cdots\ar[r] & X_0\ar[r] & N.}\]}Now take a pull-back diagram
{\footnotesize \[\xymatrix{\alpha_1:T\ar[r]\ar@{=}[d] & G\ar[r]\ar[d]_{f_1} & \syz^{n}N\ar[d]_{f} \\ \alpha':T\ar[r] & X_{n}\ar[r]& L.}\] }
Splicing the above diagrams together gives us the following commutative diagram:
{\footnotesize\[\xymatrix{\alpha_1\de_N:T\ar[r]\ar@{=}[d]& G\ar[d]\ar[r] & P_{n-1}\ar[d]\ar[r] & \cdots\ar[r] & P_0\ar[r]\ar[d] & N\ar@{=}[d]\\ \al: T\ar[r]& X_{n}\ar[r] & X_{n-1}\ar[r] & \cdots\ar[r] & X_0\ar[r] & N,}\]}as desired.
\end{proof}

\begin{rem}\label{r1}Assume that $\al\in\Ext^{n+1}_{\C}(N, T)$ is given. Take a syzygy sequence $\de_N\in\Ext^n_{\C}(N, \syz^nN)$.
According to Proposition \ref{1}, there exists a morphism $\alpha_1\de_N\lrt\al$ with fixed ends, where $\alpha_1\in\Ext^1_{\C}(\syz^nN, T)$. Now for a given element $\ga\in\Ext^n_{\C}(M, \syz^nN)$, $\alpha_1\ga$ will be an element of $\Ext^{n+1}_{\C}(M, T)$.
\end{rem}

\begin{theorem}\label{t1}With the notation of Remark \ref{r1}, the assignment
$$\Psi_{\ga}(T): \Ext^{n+1}_{\C}(N, T)\lrt\Ext^{n+1}_{\C}(M, T):~~\al\mapsto \alpha_1\ga $$
is a morphism of abelian groups.
\end{theorem}
\begin{proof}
The first issue to address is well-definedness. In this direction, assume that there exists an element  $\alpha'_1\in\Ext^1_{\C}(\syz^nN, T)$ and a morphism with fixed ends $\alpha'_1\de_N\lrt\al$. We shall prove that $\alpha_1\ga=\alpha'_1\ga$. By our construction, we have the following commutative diagram:
{\footnotesize\[\xymatrix{\de_N:\syz^{n}N\ar[d]_{f'}\ar[r] & P_{n-1}\ar[d]\ar[r] & \cdots\ar[r] & P_0\ar[r]\ar[d] & N\ar@{=}[d]\\ \alpha'':L\ar[r] & X_{n-1}\ar[r] & \cdots\ar[r] & X_0\ar[r] & N.}\]}Now Comparison Theorem (cf.  \cite[Theorem 6.16]{rot}), ensures that $f-f'$ factors through a projective object, and so, taking the following pull-back diagram
{\footnotesize \[\xymatrix{T\ar[r]\ar@{=}[d] & G''\ar[r]\ar[d] & {\syz}^{n}N\ar[d]_{f-f'} \\ \alpha':T\ar[r] & X_n\ar[r]& L,}\]}we infer that the upper row, $\alpha_1-\alpha'_1,$ splits. This, in turn, yields that $\alpha_1\ga=\alpha'_1\ga$, as desired. Next, we show that $\Psi_{\ga}(T)$ is a morphism of abelian groups. So, take elements  $\al, \be\in\Ext^{n+1}_{\C}(N, T)$. It should be proved that $\Psi_{\ga}(T)(\al)+\Psi_{\ga}(T)(\be)=\Psi_{\ga}(T)(\al+\be)$. According to Proposition \ref{1}, there are $\alpha_1, \beta_1\in\Ext^1_{\C}(\syz^nN, T)$ and morphisms of extensions $\alpha_1\de_N\lrt\al$ and $\beta_1\de_N\lrt\be$ with fixed ends. So $\alpha_1\de_N+\beta_1\de_N\lrt\al+\be$ is a morphism of extensions with fixed ends. On the other hand, by \cite[Chapter VII, Lemma 3.2]{mit},  the equality $\alpha_1\de_N+\beta_1\de_N=(\alpha_1+\beta_1)\de_N$ holds true. Thus, by our definition, one obtains the following equalities:
$$\Psi_{\ga}(T)(\al+\be)=(\alpha_1+\beta_1)\ga=\alpha_1\ga+\beta_1\ga=\Psi_{\ga}(T)(\al)+\Psi_{\ga}(T)(\be).$$
So the proof is completed.
\end{proof}

\begin{cor}\label{c2}Let  $\de_N\in\Ext^n_{\C}(N, \syz^nN)$  be a syzygy sequence. Then for any element $\ga\in\Ext^n_{\C}(M, \syz^nN)$, $\Psi_{\ga}(-):\Ext^{n+1}_{\C}(N, -)\lrt\Ext^{n+1}_{\C}(M, -)$ is a natural transformation of extension functors.
\end{cor}
\begin{proof}In view of Theorem \ref{t1}, for any object $T\in\C$, $\Psi_{\ga}(T):\Ext^{n+1}_{\C}(N, T)\lrt\Ext^{n+1}_{\C}(M, T)$ is a morphism of abelian groups. Moreover, for any morphism $g:T\rt T'$ in $\C$, one evidently has the following commutative square
\[\xymatrix{\Ext^{n+1}_{\C}(N, T)\ar[r]^{\Psi_{\ga}(T)}\ar[d]_{\Ext^{n+1}(N,g)}&\Ext^{n+1}_{\C}(M, T)\ar[d]_{\Ext^{n+1}(M, g)}\\ \Ext^{n+1}_{\C}(N, T')\ar[r]^{\Psi_{\ga}(T')}&\Ext^{n+1}_{\C}(M, T').}\]Consequently, $\Psi_{\ga}(-)$ is a natural transformation from  $\Ext^{n+1}_{\C}(N, -)$ to $\Ext^{n+1}_{\C}(M, -)$.
It should be noted that, since $\Ext^0_{\C}(M, \syz^0N)=\hom_{\C}(M, N)$, in the case $n=0$, $\ga:M\rt N$ will be a  morphism in $\C$, and also, for a given element  $\al\in\Ext^1_{\C}(N, T)$, $\al=\alpha_1$. In particular, $\Psi_{\ga}=\Ext^1(\ga, -):\Ext^{1}_{\C}(N, -)\lrt\Ext^{1}_{\C}(M, -)$ is the well-known natural transformation.
So the proof is finished.
\end{proof}

\begin{cor}\label{c3}Let $\de_N\in\Ext^n_{\C}(N, \syz^nN)$ be a syzygy sequence. Then for any object $M\in\C$, the corrspondence $$\Psi: \Ext^{n}_{\C}(M, \syz^nN)\lrt \nat(\Ext^{n+1}_{\C}(N, -), \Ext^{n+1}_{\C}(M, -)):~~\ga\mapsto \Psi_{\ga} $$ is a morphism of abelian groups.
\end{cor}
\begin{proof}
Assume that an extension $\ga\in\Ext^n_{\C}(M, \syz^nN)$ is given. By Corollary \ref{c2}, $\Psi_{\ga}$ lies in $\nat(\Ext^{n+1}_{\C}(N, -), \Ext^{n+1}_{\C}(M, -))$. Suppose that $\ga=\ga'$, as elements of $\Ext^n_{\C}(M, \syz^nN)$. We shall prove that $\Psi_{\ga}=\Psi_{\ga'}$. As mentioned in \ref{gener}(c), we may assume that, there exists a morphism of extensions $\ga\lrt\ga'$ with fixed ends. Take an arbitrary elemen $\al\in\Ext^{n+1}_{\C}(N, T)$. In view of Proposiion \ref{1}, there is a morphism $\alpha_1\de_N\lrt\al$. Thus, one may obtain a morphism of extensions $\alpha_1\ga\lrt\alpha_1\ga'$ with fixed ends, and so, $\alpha_1\ga=\alpha_1\ga'$, as objects of $\Ext^{n+1}(M, T)$, see \cite[Chapter VII, Proposition 3.1]{mit}. This means that $\Psi_{\ga}(T)(\al)=\Psi_{\ga'}(T)(\al),$ which  ensures the well-defindness of $\Psi$. Next take two elements $\ga, \ga''\in\Ext^n_{\C}(M, \syz^nN)$. So one has the equalities $$\Psi_{\ga+\ga''}(T)(\al)=\alpha_1(\ga+\ga'')=\alpha_1\ga+\alpha_1\ga''=\Psi_{\ga}(T)(\al)+\Psi_{\ga''}(T)(\al),$$ implying that $\Psi$ is a morphism of abelian groups. Hence the proof is completed.
\end{proof}

Before giving the proof of Theorem \ref{t2}, we include a couple of axuillary results. 

Recall that for any pair of objects $X, Y\in\C$, an element $\ga\in\Ext^n_{\C}(X, Y)$ lies in $\p$, if there is a morphism $f:P\rt Y$ in $\C$, with $P$ projective, and an element $\et\in\Ext^n_{\C}(X, P)$ such that $\ga=f\et$. As mentioned before, $\p$ is a subgroup of $\Ext^n_{\C}(X, Y)$, see Definition \ref{x} and Remark \ref{ppp}.

 \begin{lem}\label{p} Let $\ep: X\rt Y\st{h}\rt Z$ be an extension in $\C$ and let $\Delta_M^n: \Ext^n_{\C}(M, Z)\rt\Ext^{n+1}_{\C}(M, X)$ be the connecting homomrphism. If $\ga\in\Ext^n_{\C}(M, Z)$ lies in $\p$, then $\Delta^n_M(\ga)=0$. Moreover, the converse is true, whenever $Y$ is a projective object of $\C$.
\end{lem}
\begin{proof}
Since $\ga\in\p$, by the definition, there exists an extension $\be\in\Ext^n_{\C}(M, P)$ and a morphism $f:P\rt Z$  in $\C$ with $P\in\proj\C$, such that $\ga=f\be$. As $P$ is projective, $\ep f$ is split, and so, $(\ep f)\be$ is a zero element of $\Ext^{n+1}_{\C}(M, X).$  By \cite[Chapter VII, Proposition 3.1]{mit},  $(\ep f)\be=\ep\ga$. Now since $\Delta^n_M(\ga)=\ep\ga$, the claim follows. For the converse, assume that $\Delta^n_M(\ga)=0$. We would like to show that $\ga\in\p$.  By Theorem \ref{long}, there exists an exact sequence of abelian groups: $$\Ext^n_{\C}(M, Y)\st{\Ext^n(M, h)}\lrt\Ext^n_{\C}(M, Z)\st{\Delta^n_M}\lrt\Ext^{n+1}_{\C}(M, X).$$Since $\Delta_M^n(\ga)=0$, one finds an element  $\al\in\Ext^n_{\C}(M, Y)$ such that $h\al=\ga$, gives the desired result, because $Y\in\proj\C$. So the proof is finished.
\end{proof}

The result below reveals that any natural transformation $\Psi$ between extension functors, is determined by its action  on syzygy sequences.

\begin{prop}\label{epi1}Let  $\Phi:\Ext^{n+1}_{\C}(N, -)\rt\Ext^{n+1}_{\C}(M, -)$ be a natural transformation, and  $\de_N\in\Ext^{n+1}_{\C}(N, \syz^{n+1}N)$  an arbitrary syzygy sequence. Then for any $T\in\C$ and  $\al\in\Ext^{n+1}_{\C}(N, T)$, there exists a morphism $f:\syz^{n+1}N\rt T$ in $\C$ such that $\Phi_T(\al)=\Ext^{n+1}(M, f)\Phi_{\syz^{n+1}N}(\de_N)$.
\end{prop}
\begin{proof}
Consider a syzygy sequence $\de_N:\syz^{n+1}N\rt P_n\rt\cdots\rt P_0\rt N$.
Assume that an element $\al\in\Ext^{n+1}_{\C}(N, T)$ is given. Since each $P_i$ is projective, one may obtain the
following commutative diagram
{\footnotesize\[\xymatrix{\de_N:\syz^{n+1}N\ar[d]_{f}\ar[r] & P_{n}\ar[d]\ar[r] & \cdots\ar[r] & P_0\ar[r]\ar[d] & N\ar@{=}[d]\\ \al:T\ar[r] & X_n\ar[r]  & \cdots\ar[r] & X_0\ar[r] & N.}\]}As explained in \cite[Remark 2.7(1)]{bfss}, we may assume that $\al=f\de_N$. Consider the following commutative square:
\[\xymatrix{\Ext^{n+1}_{\C}(N, \syz^{n+1}N)\ar[r]^{\Phi_{\syz^{n+1}N}}\ar[d]_{\Ext^{n+1}(N, f)} & \Ext^{n+1}_{\C}(M, \syz^{n+1}N)\ar[d]_{\Ext^{n+1}(M, f)} \\ \Ext^{n+1}_{\C}(N, T)\ar[r]^{\Phi_T} & \Ext^{n+1}_{\C}(M, T).}\]So, we have the following equalities:  $$\Phi_T(\al)=\Phi_T\Ext^{n+1}(N, f)(\de_N)=\Ext^{n+1}(M, f)\Phi_{\syz^{n+1}N}(\de_N),$$as desired. So the proof is completed.
\end{proof}

\begin{theorem}\label{t2} Keeping the notation of Corollary \ref{c3}, the correspondence
$$\bar{\Psi}: \Ext^n_{\C}(M, \syz^nN)/{\p}\lrt\nat(\Ext^{n+1}_{\C}(N, -), \Ext^{n+1}_{\C}(M, -)):~~\bar{\ga}\mapsto \Psi_{\ga}$$is an isomorphism of abelian groups.
\end{theorem}
\begin{proof}
By Corollary \ref{c3}, $\Psi:\Ext^n_{\C}(M, \syz^nN)\lrt\nat(\Ext^{n+1}_{\C}(N, -), \Ext^{n+1}_{\C}(M, -))$, sending each $\ga$ to $\Psi_{\ga}$, is a morphism of abelian groups.  We claim that $\Psi$ is an epimorphism. To see this, take a natural transformation $\Phi:\Ext^{n+1}_{\C}(N,-)\lrt\Ext^{n+1}_{\C}(M, -)$. By making use of Proposition \ref{epi1}, we only need to show that for a given syzygy sequence $\de_N\in\Ext^{n+1}_{\C}(N, \syz^{n+1}N)$, there exists an element $\ga\in\Ext^n_{\C}(M, \syz^nN)$ such that  $\Psi_{\ga}(\de_N)=\Phi_{\syz^{n+1}N}(\de_N)$. Decompose $\de_N=\delta\de'_N$ with $\delta:\syz^{n+1}N\st{a}\rt P_n\rt\syz^nN$ and $\de'_N\in\Ext^n_{\C}(N, \syz^nN)$. So one obtains the following commutative diagram of abelian groups with exact rows:
\[\xymatrix{\Ext^n_{\C}(N, \syz^nN)\ar[r]^{\Delta^n_N} & \Ext^{n+1}_{\C}(N, \syz^{n+1}N)\ar[r]^{\epsilon}\ar[d]_{\Phi_{\syz^{n+1}N}}& \Ext^{n+1}_{\C}(N, P_n)\ar[d]_{\Phi_{P_n}} \\ \Ext^n_{\C}(M, \syz^nN)\ar[r]^{\Delta^n_M}& \Ext^{n+1}_{\C}(M, \syz^{n+1}N)\ar[r]^{\eta}& \Ext^{n+1}_{\C}(M, P_n) .}\]Since $\epsilon\Delta_N^n(\de'_N)=\epsilon(\delta\de'_N)=\epsilon(\de_N)=0$, we have $\eta\Phi_{\syz^{n+1}N}(\de_N)=0.$ Thus there is an element $\ga\in\Ext^n_{\C}(M, \syz^nN)$ such that $\Delta^n_M(\ga)=\Phi_{\syz^{n+1}N}(\de_N)$, namely, $\Phi_{\syz^{n+1}N}(\de_N)=\delta\ga$. By our definition, $\delta\ga=\Psi_{\ga}(\syz^{n+1}N)(\de_N)$, and so,  $\Phi_{\syz^{n+1}N}(\de_N)=\Psi_{\ga}(\syz^{n+1}N)(\de_N)$, as desired. Hence, it remains to show that $\ker(\Psi)=\p$.  First, one should note that by Theorem \ref{t1}, for any element $\ga\in\Ext^n_{\C}(M, \syz^nN)$ and $\al\in\Ext^{n+1}_{\C}(N, T)$, we have $\Psi_{\ga}(T)(\al)=\Delta_M^n(\ga)$, where $\Delta_M^n:\Ext^n_{\C}(M, \syz^nN)\lrt\Ext^{n+1}_{\C}(M, T)$ is the connecting homomorphism associated to $\alpha_1:T\rt G\rt\syz^nN$. This fact, in conjunction with  Lemma \ref{p}, ensures that each object in $\p$ lies in $\ker(\Psi)$. Next,  assume that $\ga\in\Ext^n_{\C}(M, \syz^nN)$ such that $\Psi_{\ga}=0$. So, for a given syzygy sequence $\de_N\in\Ext^{n+1}_{\C}(N, \syz^{n+1}N)$, $0=\Psi_{\ga}(\syz^{n+1}N)(\de_N)=\Delta_M^n(\ga)$, where $\Delta_M^n:\Ext^n_{\C}(M, \syz^nN)\lrt\Ext^{n+1}_{\C}(M, \syz^{n+1}N)$ is the connecting homomorphism associated to the extension $\syz^{n+1}N\rt P_n\rt\syz^nN$. Now the reverse implication of Lemma \ref{p} yields that $\ga\in\p$. So the proof is completed.
\end{proof}

The result below indicates that, in the case $n=0$, Theorem \ref{t2} recovers the Hilton-Rees theorem, see \cite[Theorem 7]{mart}.
\begin{cor}\label{c33}For any pair of objects $M, N\in\C$, the correspondence
$$\uhom_{\C}(M, N)\lrt\nat(\Ext^1_{\C}(N, -), \Ext^1_{\C}(M, -)):~~\bar{f}\mapsto\Ext^1(f, -)$$
is an isomorphism of abelian groups.
\end{cor}
\begin{proof}Since $\Ext^0_{\C}(M, \syz^0N)=\hom_{\C}(M, N)$, and the push-out is the composition of morphisms,  the subgroup $\p$ of  $\Ext^0_{\C}(M, \syz^0N)$, exactly contains those morphisms $M\rt N$ that factor through a projective object. That is, $\Ext^0_{\C}(M, \syz^0N)/{\p}=\uhom_{\C}(M, N)$. Now Theorem \ref{t2} finishes the proof.
\end{proof}

In the following, we will observe that if $\C=R$-$\md$, the category of all finitely presented left $R$-modules, then the quotient group in Theorem \ref{t2} can be specified as natural transformations between $\Tor$-functors. In this direction,  the Auslander-Gruson-Jensen duality plays a crucial role.  So let us recall it from \cite{aus1, gj}.

\begin{s}{\sc Auslander--Gruson-Jensen duality}\\   Assume that $R$ is an associative ring with identity, and $R$-$\md$ (resp. $\md$-$R$) is the category of finitely presented left (resp. right) $R$-modules. Assume that ${\mathsf{fp}}(R$-$\md, \mathsf{Ab})$ is the category consisting of all finitely presented additive functors $F:R$-$\md\lrt\mathsf{Ab}$, where $\mathsf{Ab}$ is the category of abelian groups. It is known that this is an abelian category with kernels and cokernels computed object-wise, see for example \cite[Section 2]{aus}. The Auslander-Gruson-Jensen duality, which was first written down explicitly by Auslander in \cite{aus1} and by Gruson
and Jensen in \cite{gj}, says that, there is an equivalence
\begin{center}
$d: {\mathsf{fp}}(R$-$\md, \mathsf{Ab})\lrt ({\mathsf{fp}}(\md$-$R, \mathsf{Ab}))^{op},$
\end{center}
sending any finitely presented functor $F:R$-$\md\rt\mathsf{Ab}$ to the functor $dF:\md$-$R\rt\mathsf{Ab}$, defined by $(dF)M=\nat(F, M\otimes_R-)$ for any finitely presented right $R$-module $M$. In particular, if $F=\hom_R(X, -)$, for some $X\in R$-$\md$, then by using Yoneda Lemma, we infer that $d(\hom_R(X, -))=-\otimes_RX$.
\end{s}
The Auslander--Gruson-Jensen duality, combined with Theorem \ref{t2}, yields the following interesting result.
\begin{cor}\label{agj}For any  finitely presented left $R$-modules $M, N$, there is an isomorphism of groups
$$\Ext^n_{\C}(M, \syz^nN)/{\p}\cong\nat(\Tor_{n+1}^R(-, M), \Tor_{n+1}^R(-, N)).$$
\end{cor}
\begin{proof}
Assume that $X$ and $Y$ are finitely presented left and right $R$-module, respectively.  Since $Y\otimes_R-$ is a right exact functor, by making use of \cite[Theorem 1.2]{hr} or \cite[Corollary 5.4]{aus}, one obtains the following isomorphism: $$\nat(\Ext^{n+1}_R(X, -), Y\otimes_R-)\cong L_{n+1}(Y\otimes_R-)(X)=\Tor_{n+1}^R(Y, X),$$which is natural in $X$ and $Y$. This, in particular, means that $d(\Ext^{n+1}_R(X, -))=\Tor_{n+1}^R(-, X)$, which is natural in $X$. Now the Auslander--Gruson-Jensen duality, together with Theorem \ref{t2} yields the desired result.
\end{proof}

Let us close this section with a couple of results. The authors thank Vitor  Gulisz for pointing these out to them.
\begin{cor}\label{vit}Let $M, N$ be finitely presented left $R$-modules. If $\Ext^i_R(M, R)=0$ for all $1\leq i\leq n$, then $\Ext^n_R(M, \syz^nN)/{\p}\cong\uhom_{R}(M, N)$.
\end{cor}
\begin{proof}Consider the following isomorphisms:
\[\begin{array}{lllll}
\Ext^n_R(M, \syz^nN)/{\p} & \cong \nat(\Tor^{R}_{n+1}(-, M), \Tor^{R}_{n+1}(-, N))\\
& \cong \nat(\Tor^{R}_1(-, \syz^nM), \Tor^R_1(-, \syz^nN))\\ &\cong\uhom_R(\syz^nM, \syz^nN)\\
&\cong\Tor_1^R(\tr\syz^nM, \syz^nN)\\
&\cong\Tor_{n+1}^R(\tr\syz^nM, N)\\
&\cong\uhom_R(M, N),
\end{array}\]
where $\tr\syz^nM$ denotes the transpose of $\syz^nM$. The first isomorphism follows from Corollary \ref{agj}, while the second and fifth ones hold true, thanks to the dimension shifting argument. The validity of the third isomorphism comes from \cite[Lemma 9]{mart}, and the fourth one is well-known, see for example \cite[Lemma 3.9]{yo} or \cite[Proposition 65]{gu}. Finally, the last isomorphism follows from \cite[Proposition 1.1.3]{iy}. So the proof is finished.
\end{proof}

Let $\Lambda$ be an artin algebra, that is, it contains a central artinian subring $R$ and $\Lambda$ is a finitely generated $R$-module. Then $D(-):=\hom_R(-, E(R/{J_R})): \Lambda$-$\md\lrt\md$-$\Lambda$ is a duality, where $J_R$ is the Jacobson radical of $R$. For any integer $n\geq 1$, $\tau_n(-):=D\tr\syz^n(-)$, is called the $n$-Auslander-Reiten translation functor. This notion has been defined by Iyama \cite{iy}, when he developed his
higher Auslander-Reiten theory.

Now \cite[Theorem 1.5]{iy} combined with Corollary \ref{vit}, immediately yields the following result.
\begin{cor}Let $\Lambda$ be an artin algebra and let $M, N$ be finitely generated $\Lambda$-modules. If $\Ext^i_{\Lambda}(M, \Lambda)=0$ for all $1\leq i\leq n$, then $\Ext^n_{\Lambda}(M, \syz^nN)/{\p}\cong D\Ext^{n+1}_{\Lambda}(N, \tau_{n+1}M).$
\end{cor}

\section{Natural transformations of extension functors over $n$-Frobenius categories }
This section contains the proof of  the main result of the paper, Theorem \ref{main2}, which concerns the natural transformations between extension functors over an $n$-Frobenius category. By the aid of this, we provide some situations in which natural transformations of (co)homology functors over a ring form a set. Let us  begin  by recalling the following useful remark.

\begin{rem}\label{lr}Assume that $\C$ is an $n$-Frobenius category and $X, Y$ are two objects of $\C$. Then a given extension $\ga\in\Ext^n_{\C}(X, Y)$ lies in $\p$, if there is a morphism $f:P\rt Y$ in $\C$, where $P$ is $n$-projective, and an extension $\be\in\Ext^n_{\C}(X, P)$ such that $\ga=f\be$, see Remark \ref{ppp}. It is worthwhile to note that since $\C$ is $n$-Frobenius, \cite[Proposition 4.3]{bfss} indicates that $\ga\in\p$ if and only if there exists a morphism $g:X\rt Q$ in $\C$ with $Q$ $n$-projective, and $\et\in\Ext^n_{\C}(Q, Y)$ such that $\ga=\et g$.
\end{rem}

\begin{lem}\label{pp}Let $\C$ be an $n$-Frobenius category and let
$\ga\in\Ext^n_{\C}(X, Y)$ be a $\p$-extension. Then for any extension $\alpha\in\Ext^1_{\C}(Y, Z)$, the composition $\alpha\ga=0$ in $\Ext^{n+1}_{\C}(X, Z)$.
\end{lem}
\begin{proof}Since $\ga$ is a $\p$-extension, as mentioned in Remark \ref{lr}, there exists a morphism $h:X\rt P$ in $\C$, where $P$ is $n$-projective, and an element $\be\in\Ext^n_{\C}(P, Y)$ such that $\ga=\be h$. As $P$ is $n$-projective, $\Ext^{n+1}_{\C}(P, -)=0$. So, for a given extension $\alpha\in\Ext^1_{\C}(Y, Z)$, we have that $\alpha\be=0$, implying that $(\alpha\be)h=\alpha(\be h)=\alpha\ga=0$, as desired.
\end{proof}

\begin{rem}Assume that $\C$ has enough projective objects. Then the lifting property yields that for any object $N\in\C$ and an integer $n\geq 1$, the $n$-th syzygy $\syz^nN$ of $N$ is unique, up to projective summands. However, this is not the case for $n$-Frobenius categories, where $\syz^nN$ is obtained using $n$-projective objects. Our key strategy to fix this flaw is the use of (co)angled pairs. Indeed, as observed in \ref{ss}(2), for any pair of syzygy sequence $\de_{N}, \de'_N\in\U_n(N)$, with respect to $n$-projectives, there is an angled pair $\de_N\st{a}\rt\de''_N\st{b}\lf\de'_{N}$. In particular, there are morphisms $\syz^nN\st{a}\rt{\syz''}^nN\st{b}\lf{\syz'}^nN$, which are inflations with $n$-projective cokernels, and so, they are $n$-$\Ext$-invertible morphisms.
\end{rem}

In the remainder of the paper, we always assume that $\C$ is an $n$-Frobenus category, with $n$ a non-negative integer. {Moreover, syzygy sequences (which are obtained using $n$-projectives), will  be also called unit extensions.}

\begin{prop}\label{ll}For a given element  $\al\in\Ext^{n+1}_{\C}(N, T)$, there exists a unit extension $\de_N\in\Ext^n_{\C}(N, \syz^nN)$ and an extension $\alpha_1\in\Ext^1_{\C}(\syz^nN, T)$ such that there is a morphism $\alpha_1\de_N\rt\al$ with fixed ends. Moreover, if $\be\in\Ext^{n+1}_{\C}(N, T)$ is another extension, then there exists a unit extension $\de'_N\in\Ext^n_{\C}(N, {\syz'}^nN)$ and morphisms $\alpha_1\de'_N\rt\al$ and $\beta_1\de'_N\rt\be$ with fixed ends.
\end{prop}
\begin{proof}Decompose $\al=\alpha_n\al'$ with $\alpha_n: T\rt X_n\rt L$ and $\al': L\rt X_{n-1}\rt\cdots\rt X_0\rt N$. In view of \cite[Lemma 5.2(2)]{bfss},   one has the following commutative diagram with exact rows:
{\footnotesize\[\xymatrix{\de_N:{\syz}^{n}N\ar[d]_{f}\ar[r] & P_{n-1}\ar[d]\ar[r] & \cdots\ar[r] & P_0\ar[r]\ar[d] & N\ar@{=}[d]\\ \al':L\ar[r] & X_{n-1}\ar[r] & \cdots\ar[r] & X_0\ar[r] & N,}\]}where each $P_i$ is $n$-projective. Setting $\alpha_1:=\alpha_nf$, one obtains the morphism of extensions $\alpha_1\de_N\lrt\al$ with fixed ends. For the second assertion, decompose $\be=\beta_n\be'$ with $\beta_n:T\rt Y_n\rt L'$ and $\be': L'\rt Y_{n-1}\rt\cdots\rt Y_0\rt N$. Set $\et:=(\al'\oplus\be')\Delta$, where $\Delta:N\st{{{\tiny {\left[\begin{array}{ll} 1 \\ 1 \end{array} \right]}}}}\lrt N\oplus N$ is a morphism in $\C$. Another use of \cite[Lemma 5.2(2)]{bfss}, yields the following commutative diagram:
{\footnotesize\[\xymatrix{\de'_N:{\syz'}^{n}N\ar[d]_{{{{\tiny {\left[\begin{array}{ll} g \\ g' \end{array} \right]}}}}}\ar[r] & P'_{n-1}\ar[d]\ar[r] & \cdots\ar[r] & P'_0\ar[r]\ar[d] & N\ar@{=}[d]\\ \et:L\oplus L'\ar[r] & T_{n-1}\ar[r] & \cdots\ar[r] & T_0\ar[r] & N.}\]}Now putting $\alpha_1:=\alpha_ng$ and $\beta_1:=\beta_ng'$, one may obtain the morphism of extensions $\alpha_1\de'_N\lrt\al$ and $\beta_1\de'_N\lrt\be$ with fixed ends. So the proof is finished.
\end{proof}

\begin{s}Assume that $a:M\rt M'$ is an $n$-$\Ext$-invertible morphism in $\C$ and $\ga\in\Ext^n_{\C}(M, N)$ is given. So  one may find an element $\ga'\in\Ext^n_{\C}(M', N)$ such that $\ga-\ga'a$ is a $\p$-extension. To simplify the notation, we put $\ga':=_{_{\p}}\ga a^{-1}$. Similarly, for any element $\be\in\Ext^n_{\C}(N, M')$, there is  $\be'\in\Ext^n_{\C}(N, M)$ such that $\be -a\be'$ is a $\p$-extension. In this case, we write $\be':=_{_{\p}}a^{-1}\be$.
\end{s}

\begin{rem}\label{3s1}Assume that a unit extension $\de_N\in\Ext^n_{\C}(N, \syz^nN)$ and  $\al\in\Ext^{n+1}_{\C}(N, T)$ are given. According to Proposition \ref{ll}, there exists a unit extension $\de'_N\in\Ext^{n}_{\C}(N, {\syz'}^nN)$ and a morphism of extensions $\alpha_1\de'_N\lrt\al$ with fixed ends, in which $\alpha_1\in\Ext^1_{\C}({\syz'}^nN, T)$.
Now for a given element $\ga\in\Ext^n_{\C}(M, \syz^nN)$, we infer that $b^{-1}a\ga\in\Ext^n_{\C}(M, {\syz'}^nN)$, where $\de_N\st{a}\rt\de''_N\st{b}\lf\de'_N$ is an angled pair. In particular, setting $\ga':=b^{-1}a\ga$, we have that $\alpha_1\ga'\in\Ext^{n+1}_{\C}(M, T)$.
\end{rem}
\begin{prop}\label{3p1}With the notation of  Remark \ref{3s1}, $$\Psi_{\ga}(T):\Ext^{n+1}_{\C}(N, T)\lrt\Ext^{n+1}_{\C}(M, T):~~ \al\mapsto\alpha_1\ga'$$ is a morphism of abelian groups.
\end{prop}
\begin{proof}First we prove that $\Psi_{\ga}(T)$ is well-defined. To do this, assume that there exists a unit extension $\de_{1N}\in\Ext^n_{\C}(N, \syz^n_1N)$ and a morphism $f':\syz^n_1N\rt L$ in which $(\alpha_nf')\de_{1N}\lrt\al$ is a morphism of extensions with fixed ends. Now considering an angled pair $\de_N\st{a_1}\rt\de_1\st{b_1}\lf\de_{1N}$, we prove that $(\alpha_nf)b^{-1}a\ga=_{\p}(\alpha_nf')b_1^{-1}a_1\ga$. One should note that, since there exists a morphism of extensions $(\alpha_nf)b^{-1}a\ga\lrt\alpha_n(f(b^{-1}a\ga))$ with fixed ends, by \cite[Chapter VII, Proposition 3.1]{mit}, $(\alpha_nf)b^{-1}a\ga=\alpha_n(f(b^{-1}a\ga))$. Analoguesly,  $(\alpha_nf')b_1^{-1}a_1\ga=\alpha_n(f'(b_1^{-1}a_1\ga))$. So it suffices to show that  $\alpha_n(f(b^{-1}a\ga))=\alpha_n(f'(b_1^{-1}a_1\ga))$. To this end, similar to the proof of  \cite[Theorem 7.4]{bfss}, one obtains the following diagram of angled pairs:
{\footnotesize \[\xymatrix{&\delta_N\ar[dl]_{a}\ar[dr]^{a_1}& \\ \delta''_{N}~\ar[r]^{a_3}& \delta_3& \delta_1\ar[l]_{b_3}\\ \delta'_{N}\ar[u]^{b}\ar[r]^{a_2}~ & \delta_2\ar[u]^{c_3} & \delta_{1N}\ar[l]_{b_2}\ar[u]^{b_1}.}\]}By making use of \cite[Proposition 5.8]{bfss}, there is a morphism $h$ in $\C$ such that $\ga=h\de_2$, $f=ha_2$ and $f'=hb_2$. In particular, one may have the following diagram in $\C$: {\footnotesize\[\xymatrix{\de'_N~ \ar[r]^{a_2}\ar[rd]_{f}& \de_2\ar[d]^{h}& {\de_1}_{N}\ar[l]_{b_2}\ar[ld]^{f'}\\ & \ga .& }\]}
Thus, we have the following equalities modulo $\p$: $$f'(b_1^{-1}(a_1\ga))=_{\p}hb_2(b_1^{-1}(a_1\ga))=_{\p}h(b_2(b_1^{-1}(a_1\ga)))=_{\p}fa_2^{-1}(b_2(b_1^{-1}(a_1\ga)))$$ $$=_{\p}f(a_2^{-1}(b_2(b_1^{-1}(a_1\ga))))=_{\p} fa_2^{-1}(c_3^{-1}b_3)(a_1\ga)))=_{\p}f(a_2^{-1}c_3^{-1})(b_3(a_1\ga))$$$$=_{\p}f(b^{-1}a_3^{-1})(b_3(a_1\ga))=_{\p}f(b^{-1}(a_3^{-1}b_3a_1)\ga)=_{\p}f(b^{-1}(a\ga)).$$Now Lemma \ref{pp} together with Remark \ref{invert}, ensures that $\Psi_{\ga}(T)$ is well-defined.
Finally, we show that $\Psi_{\ga}(T)$ is a morphism of abelian groups. To this end, consider extensions $\al, \be\in\Ext^{n+1}_{\C}(N, T)$. In view of Proposition \ref{ll}, there are morphisms of extensions $\alpha_1\de'_N\lrt\al$ and $\beta_1\de'_N\lrt\be$ with fixed ends, for some unit extension $\de'_N\in\Ext^n_{\C}(N, {\syz'}^nN)$. This, in conjunction with \cite[Chapter VII, Lemma 3.2]{mit}, gives us the morphism of extensions $(\alpha_1+\beta_1)\de'_N\lrt\al+\be$ with fixed ends. Now putting $\ga':=b^{-1}a\ga$, where $\de_N\st{a}\rt\de''_N\st{b}\lf\de'_N$ is an angled pair, one may obtain the following equalities:
$$\Psi_{\ga}(T)(\al+\be)=(\alpha_1+\beta_1)\ga'=\alpha_1\ga'+\beta_1\ga'=\Psi_{\ga}(T)(\al)+\Psi_{\ga}(T)(\be).$$Thus the proof is completed.
\end{proof}

The result below is proven similar to that of Corollary \ref{c3}. However, for the reader's convenience, its proof is included.
\begin{prop}\label{p33}Let $\de_N\in\Ext^n_{\C}(N, \syz^nN)$ be a unit extension. Then for any object $M\in\C$, the correspondence
$$\Psi:\Ext^n_{\C}(M, \syz^nN)\lrt\nat(\Ext^{n+1}_{\C}(N, -), \Ext^{n+1}_{\C}(M, -)) ~~: ~~\ga\mapsto \Psi_{\ga}$$
is a morphism of abelian groups. 
\end{prop}
\begin{proof}First we show that $\Psi$ is well-defined. In this direction, take an object $T\in\C$. By  Proposition \ref{3p1}, $\Psi_{\ga}(T):\Ext^{n+1}_{\C}(N, T)\lrt\Ext^{n+1}_{\C}(M, T)$ is a morphism of abelian groups. In particular, one may follow the argument given in the proof of Corollary \ref{c2} and deduce that $\Psi_{\ga}:\Ext^{n+1}_{\C}(N, -)\lrt\Ext^{n+1}_{\C}(M, -)$ is a natural transformation. Next assume that $\ga$ and $\ga'$ are two elements of $\Ext^n_{\C}(M, \syz^nN)$ such that there exists a morphism of extensions $\ga\lrt\ga'$ with fixed ends. For a given object $T\in\C$, we shall prove that $\Psi_{\ga}(T)=\Psi_{\ga'}(T)$. In this direction, take an arbitrary element $\al\in\Ext^{n+1}_{\C}(N, T)$. In view of Proposition \ref{ll}, there exists a unit extension $\de'_N\in\Ext^n_{\C}(N, {\syz'}^nN)$ and $\alpha_1\in\Ext^1_{\C}({\syz'}^nN, T)$, in which there is a morphism of extensions $\alpha_1\de'_N\lrt\al$ with fixed ends. Consider an angled pair $\de_N\st{a}\rt\de''_N\st{b}\lf\de'_N$. Since there is a morphism of extensions $\ga\lrt\ga'$ with fixed ends, using the universal property of push-out, one may find a morphism of extensions $b^{-1}a\ga\lrt b^{-1}a\ga'$ with fixed ends. Thus, $\alpha_1(b^{-1}a\ga)\lrt\alpha_1( b^{-1}a\ga')$ will also be a morphism of extensions with fixed ends. Consequently, as elements of $\Ext^{n+1}_{\C}(M, T)$, these extensions are the same, meaning that  $\Psi_{\ga}(T)(\al)=\Psi_{\ga'}(T)(\al)$, as desired. Thus $\Psi$ is well-defined. Next, we show that $\Psi$ is a morphism of groups. To this end, considering a pair of extensions $\ga, \ga''\in\Ext^n_{\C}(M, \syz^nN)$, one gets the following equalities: $$\Psi_{\ga+\ga''}(T)(\al)=\alpha_1(b^{-1}a(\ga+\ga''))=\alpha_1(b^{-1}a\ga)+b^{-1}a\ga'')=\alpha_1(b^{-1}a\ga)+\alpha_1(b^{-1}a\ga'').$$ But the right-hand side equals to $\Psi_{\ga}(T)(\al)+\Psi_{\ga''}(T)(\al)$, and so,  the proof is finished.
\end{proof}


\begin{prop}\label{epi}Let  $\Phi:\Ext^{n+1}_{\C}(N, -)\rt\Ext^{n+1}_{\C}(M, -)$ be a natural transformation and let  $\de_N\in\Ext^{n+1}_{\C}(N, \syz^{n+1}N)$ be a unit extension. Then for any object $T\in\C$ and  $\al\in\Ext^{n+1}_{\C}(N, T)$, there exists a morphism $f:{\syz'}^{n+1}N\rt T$ in $\C$ and $n$-$\Ext$-invertible morphisms  $\syz^{n+1}N\st{a}\rt{\syz''}^{n+1}N\st{b}\lf{\syz'}^{n+1}N$ such that $\Phi_T(\al)=f^*{b^*}^{-1}a^*\Phi_{\syz^{n+1}N}(\de_N)$, where $(-)^*=\Ext^{n+1}(M, -)$.
\end{prop}
\begin{proof}
Assume that $T\in\C$ and $\al\in\Ext^{n+1}_{\C}(N, T)$ is given. Take an $\luf$ $\al=f\de'_N$ of $\al$, where $f:{\syz'}^{n+1}N\rt T$ is a morphism in $\C$. Consider an angled pair $\de_N\st{a}\rt\de''_N\st{b}\lf\de'_N$. Now we obtain the following commutative square:
\[\xymatrix{\Ext^{n+1}_{\C}(N, \syz^{n+1}N)\ar[r]^{\Phi_{\syz^{n+1}N}}\ar[d]_{f^*{b^*}^{-1}a^*} & \Ext^{n+1}_{\C}(M, \syz^{n+1}N)\ar[d]_{f^*{b^*}^{-1}a^*} \\ \Ext^{n+1}_{\C}(N, T)\ar[r]^{\Phi_T} & \Ext^{n+1}_{\C}(M, T).}\]
One should note that $(-)^*$ in the left (resp. right) column stands for $\Ext^{n+1}(N, -)$ (resp. $\Ext^{n+1}(M, -)$). So, going in the counterclockwise side, one gets the equalities $\Phi_T f^*{b^*}^{-1}a^*(\de_N)=\Phi_Tf^*(\de'_N)=\Phi_T(\al)$. Now since the square is commutative, we have that $\Phi_T(\al)=f^*{b^*}^{-1}a^*\Phi_{\syz^{n+1}N}(\de_N)$, where $(-)^*=\Ext^{n+1}(M, -)$. So the proof is finished.
\end{proof}


Now we are ready to prove the main result of the paper.

\begin{theorem}\label{nat} Let $M, N$ be arbitrary objects of $\C$ and let $\de_N\in\Ext^n_{\C}(N, \syz^nN)$ be a unit extension. Then the correspondence
$$\Psi: \Ext^n_{\C}(M, \syz^nN)\lrt{\nat}(\Ext^{n+1}_{\C}(N, -), \Ext^{n+1}_{\C}(M, -))~:~~\ga\mapsto \Psi_{\ga}$$
induces an isomorphism
$$\Ext^n_{\C}(M, \syz^nN)/{\p}\st{\cong}\lrt{\nat}(\Ext^{n+1}_{\C}(N, -), \Ext^{n+1}_{\C}(M, -)).$$
\end{theorem}
\begin{proof}
As observed in Proposition \ref{p33}, for any object $M\in\C$,  $$\Psi:\Ext^n_{\C}(M, \syz^nN)\lrt\nat(\Ext^{n+1}_{\C}(N,-), \Ext^{n+1}_{\C}(M, -))~:~~\ga\mapsto\Psi_{\ga}$$ is a morphism of abelian groups. We would like to show that $\Psi$ is an epimorphism. In this direction, take a natural transformation $\Phi:\Ext^{n+1}_{\C}(N,-)\lrt\Ext^{n+1}_{\C}(M, -)$. Consider a unit extension $\be=\delta\de_N\in\Ext^{n+1}_{\C}(N, \syz^{n+1}N)$ with $\delta:\syz^{n+1}N\rt P_n\rt\syz^nN$. So one obtains the following commutative diagram with exact rows:\[\xymatrix{\Ext^n_{\C}(N, \syz^nN)\ar[r]^{\Delta^n_N} & \Ext^{n+1}_{\C}(N, \syz^{n+1}N)\ar[r]^{\epsilon}\ar[d]_{\Phi_{\syz^{n+1}N}}& \Ext^{n+1}_{\C}(N, P_n)\ar[d]_{\Phi_{P_n}} \\ \Ext^n_{\C}(M, \syz^nN)\ar[r]^{\Delta^n_M}& \Ext^{n+1}_{\C}(M, \syz^{n+1}N)\ar[r]^{\eta}& \Ext^{n+1}_{\C}(M, P_n),}\]where $\Delta^n:\Ext^n_{\C}(-, \syz^nN)\lrt\Ext^{n+1}_{\C}(-, \syz^{n+1}N)$ is the connecting homomorphism associated to $\delta:\syz^{n+1}N\rt P_n\rt\syz^nN$. {One should note that since $P_n$ is $n$-projective,  it is an $n$-injective object, implying that $\Ext^{n+1}_{\C}(-, P_n)=0$. This, in particular, yields that $\Delta^n_M$ is an epimorphism.  Hence, }  one may find an element $\ga\in\Ext^n_{\C}(M, \syz^nN)$ such that $\Delta_M^n(\ga)=\Phi_{\syz^{n+1}N}(\be)$. On the other hand, since $\delta\de_N\st{=}\lrt\be$ is a morphism with fixed ends, by our definition, $\Psi_{\ga}(\syz^{n+1}N)(\be)=\delta\ga=\Delta_M^n(\ga)$. Consequently, $\Psi_{\ga}(\syz^{n+1}N)(\be)=\Phi_{\syz^{n+1}N}(\be).$ Now Proposition \ref{epi} ensures that $\Psi$ is an epimorphism. Hence, in order to complete the proof, we need to show that $\ker\Psi=\p$. To this end, assume that $\ga\in\Ext^n_{\C}(M, \syz^nN)$ is a $\p$-extension. Then, by making use of  Lemma \ref{pp} and Remark \ref{invert}, one may infer that $\Psi_{\ga}=0$. Conversely, assume that $\ga\in\ker\Psi$. So, taking $\al:=\delta\de_N$, with $\delta:\syz^{n+1}N\rt P_n\rt\syz^nN$, we have $\Psi_{\ga}(\syz^{n+1}N)(\al)=\delta\ga=0$. Thus, $\Delta^n_M(\ga)=\delta\ga=0$, where $\Delta_M^n:\Ext^n_{\C}(M, \syz^nN)\lrt\Ext^{n+1}_{\C}(M, \syz^{n+1}N)$ is the connceting homomorphism associted to $\delta$. Now, since $P_n$ is $n$-projective,  the same argument given in the proof of the reverse implication in Lemma \ref{p}, yields that $\ga\in\p$. Hence, the proof is completed.
\end{proof}

\begin{rem}\label{rr}Assume that $M, T$ are objects in $\C$ and $\al\in\Ext^{n+1}_{\C}(T, M)$ is given. Decompose $\al=\al'\alpha'_1$ with $\al':M\rt X_n\rt\cdots\rt X_1\rt L$ and $\alpha'_1:L\rt X_0\rt T$. In view of \cite[Lemma 5.2(1)]{bfss},   one has the following commutative diagram with exact rows:
{\footnotesize\[\xymatrix{\al':~~M\ar@{=}[d]\ar[r] & X_n\ar[d]\ar[r] & \cdots\ar[r] & X_1\ar[r]\ar[d] & L\ar[d]_{f}\\ \de_{\syz^{-n}M}:M\ar[r] & P_{n}\ar[r] & \cdots\ar[r] & P_1\ar[r] & \syz^{-n}M,}\]}where each $P_i$ is $n$-projective. Now, putting $\alpha_1:=f\alpha'_1$, one obtains the morphism of extensions $\al\lrt\de_{\syz^{-n}M}\alpha_1$ with fixed ends.
\end{rem}

\begin{s}\label{cov}Assume that $M, N$ are objects of $\C$ and $\de_{\syz^{-n}M}\in\Ext^n_{\C}(\syz^{-n}M, M)$ is a unit extension. Then for a given element $\ga\in\Ext^n_{\C}(\syz^{-n}M, N)$, we assign a natural transformation $\Phi_{\ga}:\Ext^{n+1}_{\C}(-, M)\lrt\Ext^{n+1}_{\C}(-, N)$,  as follows:\\ For a given object $T\in\C$, we take an element $\al\in\Ext^{n+1}_{\C}(T, M)$. As mentioned in Remark \ref{rr}, there is a morphism of extensions $\al\lrt\de_{{\syz'}^{-n}M}\alpha_1$ with fixed ends, where $\de_{{\syz'}^{-n}M}\in\Ext^n_{\C}({\syz'}^{-n}M, M)$ is a unit extension and $\alpha_1\in\Ext^1_{\C}(T, {\syz'}^{-n}M)$. So considering a co-angled pair $\de_{\syz^{-n}M}\st{a}\lf\de_{{\syz''}^{-n}M}\st{b}\rt\de_{{\syz'}^{-n}M}$, one has $(\ga a)b^{-1}\in\Ext^n_{\C}({\syz'}^{-n}M, N)$. Now we define $\Phi_{\ga}(T)(\al):=((\ga a)b^{-1})\alpha_1$, which lies in $\Ext^{n+1}_{\C}(T, N)$. A dual manner given in the proof of Proposition \ref{p33} and using Remark \ref{lr}, reveals that $\Phi_{\ga}:\Ext^{n+1}_{\C}(-, M)\lrt\Ext^{n+1}_{\C}(-, N)$ is a natural transformation.  In particular, as $\C$ has enough $n$-injectives and $n$-injectives coincide with $n$-projectives,   dualizing the argument, yields the result below.
\end{s}

\begin{theorem}\label{th3}With the notation above, the correspondence
$$\Phi:\Ext^n_{\C}(\syz^{-n}M, N)\lrt\nat(\Ext^{n+1}_{\C}(-, M), \Ext^{n+1}_{\C}(-, N))~:\ga\mapsto\Phi_{\ga}$$ induces an isomorphism of abelian groups
$$\bar{\Phi}:\Ext^n_{\C}(\syz^{-n}M, N)/{\p}\st{\cong}\lrt\nat(\Ext^{n+1}_{\C}(-, M), \Ext^{n+1}_{\C}(-, N))~:\bar{\ga}\mapsto\Phi_{\ga}.$$ 
\end{theorem}

\begin{cor}\label{cc}Let $M, N$ be objects of $\C$ and $\de_N\in\Ext^n_{\C}(N, \syz^nN)$ a unit extension. Then there is an isomorphism $$\Ext^n_{\C}(M, \syz^nN)/{\p}\cong\nat(\Ext^{n+1}_{\C}(-, M), \Ext^{n+1}_{\C}(-, N)).$$
\end{cor}
\begin{proof}
In view of Theorem \ref{th3}, it suffices to show that for a given unit extension $\de_{\syz^{-n}M}\in\Ext^n_{\C}(\syz^{-n}M, M)$, an isomorphism  $\Ext^n_{\C}(\syz^{-n}M, N)/{\p}\cong\Ext^n_{\C}(M, \syz^nN)/{\p}$ holds true. To do this, assume that $\de_N:=\syz^nN\rt P_{n-1}\rt\cdots\rt P_0\rt N$. Since $P_0$ is $n$-projective, and so, it is an $n$-injective object, $\Ext^{n+1}_{\C}(-, P_0)=0$. Thus, we give an exact sequence $$\Ext^n_{\C}(\syz^{-n}M, P_0)\lrt\Ext^n_{\C}(\syz^{-n}M, N)\lrt\Ext^{n+1}_{\C}(\syz^{-n}M, \syz N)\lrt 0.$$This yields an isomorphism $\Ext^n_{\C}(\syz^{-n}M, N)/{\p}\cong\Ext^{n+1}_{\C}(\syz^{-n}M, \syz N)$. Similarly, considering the unit extension $\syz^{-n+1}M\rt Q_n\rt\syz^{-n}M$, and applying Remark \ref{lr}, one gets an isomorphism $\Ext^n_{\C}(\syz^{-n+1}M, \syz N)/{\p}\cong \Ext^{n+1}_{\C}(\syz^{-n}M, \syz N)$. Combining the latter two isomorphisms would give us the isomorphism of groups $\Ext^n_{\C}(\syz^{-n}M, N)/{\p}\cong\Ext^n_{\C}(\syz^{-n+1}M, \syz N)/{\p}$. Now one may continue this argument repeatedly, and deduce that $\Ext^n_{\C}(\syz^{-n}M, N)/{\p}\cong\Ext^n_{\C}(M, \syz^n N)/{\p}$, as desired.
\end{proof}

Since $\C$ is an $n$-Frobenius category, as noted in \cite[Remark 2.4]{bfss}, it is a $k$-Frobenius category with $n$-$\proj\C=k$-$\proj\C$, for all integers $k>n$. So, as a direct consequence of Theorem \ref{nat} and Corollary \ref{cc}, we include the following result.
\begin{cor}\label{ccc}Let $M, N$ be objects of $\C$ and $k\geq n$ an integer, and let $\de_N\in\Ext^k_{\C}(N, \syz^kN)$ be a unit extension. Then one has the following isomorphisms of abelian groups
\[\begin{array}{lllll}
\Ext^k_{\C}(M, \syz^kN)/{\p} & \cong \nat(\Ext^{k+1}_{\C}(N, -), \Ext^{k+1}_{\C}(M, -))\\
& \cong \nat(\Ext^{k+1}_{\C}(-, M), \Ext^{k+1}_{\C}(-, N)).
\end{array}\]
\end{cor}
The above corollary immediately yields the following interesting result.
\begin{cor}For any pair of objects $M, N\in\C$ and an integer $k\geq n$,  the groups $\nat(\Ext^{k+1}_{\C}(M, -), \Ext^{k+1}_{\C}(N, -))$ and $\nat(\Ext^{k+1}_{\C}(-, M), \Ext^{k+1}_{\C}(-, N))$ form a set.
\end{cor}

Since over Frobenius categories (or equivalently, 0-Frobenius categories), $\p$-extensions are those morphisms that factor through a projective object, the case $n=0$ of Theorem \ref{nat} and Corollary \ref{cc} recovers Theorems 7 and 8 of \cite{mart}, which is the Hilton-Rees theorem for Frobenius categories. Indeed, we have the result below.

\begin{cor}\label{c44}Let $\C$ be a Frobenius category. Then  for any two objects $M, N\in\C$, one has the following isomorphisms of abelian groups
\[\begin{array}{lllll}
\underline{\hom}_{\C}(M, N) & \cong \nat(\Ext^{1}_{\C}(N, -), \Ext^{1}_{\C}(M, -))\\
& \cong \nat(\Ext^{1}_{\C}(-, M), \Ext^{1}_{\C}(-, N)).
\end{array}\]
\end{cor}
We close this paper by presenting several interesting examples of $n$-Frobenius categories that satisfy the  isomorphisms appeared in Corollary \ref{ccc}.

\begin{example}\label{sep}Assume that $(\x, \O_{\x})$ is a semi-separated noetherian scheme  of finite Krull dimension. As noted  in Example \ref{xx1}, the category $\C(\L)$ consisting of all syzygies of complete resolutions of locally free sheaves of finite, is an $n$-Frobenius subcategory of $\coh(\x)$, such that $n$-$\proj\C(\L)=\L$. So, for any pair of objects $X, Y\in\C(\L)$ and $k\geq n$, the following isomorphisms hold:
\[\begin{array}{lllll}
\Ext^k_{\C(\L)}(X, \syz^kY)/{\p} & \cong \nat(\Ext^{k+1}_{\C(\L)}(Y, -), \Ext^{k+1}_{\C(\L)}(X, -))\\
& \cong \nat(\Ext^{k+1}_{\C(\L)}(-, X), \Ext^{k+1}_{\C(\L)}(-, Y)),
\end{array}\]
where $\syz^kY$ is obtained by using locally free sheaves of finite rank.
\end{example}

\begin{example}Let $(R, \m)$ be a commutative Gorenstein local ring with $\dim R=d$.\\
(1) Since by \cite[Theorem 10.2.4]{ej}, $d$-th syzygy of any $R$-module is Gorenstein projective, one may see that the category $R$-$\Md$ of all $R$-modules is a $d$-Frobenius category, where $d$-projectives are modules of finite projective dimension. So, for any pair of $R$-modules $M, N$ and $k\geq d$, the following isomorphisms of groups
\[\begin{array}{lllll}
\Ext^k_R(M, \syz^kN)/{\p} & \cong \nat(\Ext^{k+1}_R(N, -), \Ext^{k+1}_R(M, -))\\
& \cong \nat(\Ext^{k+1}_R(-, M), \Ext^{k+1}_R(-, N)),
\end{array}\]
hold. We should emphasize that $\syz^kN$ is obtained by using $d$-projectives. Furthermore, if $M$ and $N$ are assumed to be finitely generated modules, then  the Auslander-Gruson-Jensen duality would imply that these will be isomorphic to the natural transformations $\nat(\Tor_{k+1}^R(M, -), \Tor_{k+1}^R(N, -))$ of homological functors.
\\ (2)  Assume that $\Q:\bullet\lrt\bullet$ is a quiver. Then the path algebra $R\Q$ is a Gorenstein order in the sense of Auslander \cite{auslander1967functors}. Assume that $\G$ is the category consisting of all lattices over $R\Q$, that is, those $R\Q$-modules $M$ which are finitely generated Gorenstein projective over $R$. As noted in  Example \ref{ex}, $\G$ will be a 1-Frobenius subcategory of $R\Q$-$\Md$.
So for any $X, Y\in\G$ and $i\geq 1$, the following isomorphisms hold
\[\begin{array}{lllll}
\Ext^i_{\G}(X, \syz^iY)/{\p} & \cong \nat(\Ext^{i+1}_{\G}(Y, -), \Ext^{i+1}_{\G}(X, -))\\
& \cong \nat(\Ext^{i+1}_{\G}(-, X), \Ext^{i+1}_{\G}(-, Y)).
\end{array}\]
These will be also isomorphic to the group $\nat(\Tor_{i+1}^{R\Q}(X, -), \Tor_{i+1}^{R\Q}(Y, -))$, whenever $X, Y$ are finitely generated modules.\\ (3) Assume that $R$ is complete and $d>0$. As mentioned in Example \ref{ex}(3), the category $\C$ of all artinian $R$-modules is a $d$-Frobenius category, without having projective objects. Then for any pair of objects $M, N\in\C$ and $k\geq d$, one has the following isomorphisms:  
\[\begin{array}{lllll}
\Ext^k_{\C}(M, \syz^kN)/{\p} & \cong \nat(\Ext^{k+1}_{\C}(N, -), \Ext^{k+1}_{\C}(M, -))\\
& \cong \nat(\Ext^{k+1}_{\C}(-, M), \Ext^{k+1}_{\C}(-, N)).
\end{array}\]

\end{example}




\end{document}